\documentclass[12pt]{amsart}

\usepackage{amsmath, amscd, amssymb}

\newcommand{\bC}{{\mathbb C}}

 \newcommand{\cD}{{\mathcal D}}
\newcommand{\cW}{{\mathcal W}} \newcommand{\tW}{\tilde{\mathcal W}}

\newcommand{\Mbar}{\overline{\mathcal M}}

\newcommand{\pd}{\partial}
\newcommand{\half}{\frac{1}{2}}

\newcommand{\cor}[1]{\langle {#1} \rangle}
\DeclareMathOperator{\res}{Res}

\newtheorem{theorem}{Theorem}[section]

\newtheorem{lemma}[theorem]{Lemma}

\theoremstyle{remark}
\newtheorem{remark}{Remark}[section]

\theoremstyle{definition}

\newcommand{\bea}{\begin{eqnarray}}
\newcommand{\eea}{\end{eqnarray}}
\newcommand{\ben}{\begin{eqnarray*}}
\newcommand{\een}{\end{eqnarray*}}
\newcommand{\be}{\begin{equation}}
\newcommand{\ee}{\end{equation}}

\begin{document}

\title[Intersection Numbers and Quantum Airy Curve]
{Intersection Numbers on Deligne-Mumford Moduli Spaces and Quantum Airy Curve}

\author{Jian Zhou}
\address{Department of Mathematical Sciences\\Tsinghua University\\Beijing, 100084, China}
\email{jzhou@math.tsinghua.edu.cn}

\begin{abstract}
We establish the Airy curve case of a conjecture of Gukov and Su{\l}kowski
by reducing to Dijkgraaf-Verlinde-Verlinde Virasoro constraints
satisfied by the intersection numbers on moduli spaces of algebraic curves.
\end{abstract}

\maketitle

{\bf Key words.} Intersection numbers, moduli spaces of curves,
Eynard-Orantin topological recursion.

 {\bf MSC 2000.} Primary 14N35.  Secondary 53D45.
\section{Introduction}

In this paper we will establish the Airy curve case
of a conjecture of Gukov and Su{\l}kowski \cite{Guk-Sul}.
By the Airy curve we mean the plane algebraic curve defined by the following equation:
\be
A(u,v) = \frac{1}{2} v^2 - u = 0.
\ee
(This differs from the form used in \cite{Guk-Sul} by a factor of $2$.)
This curve has the following parametrization:
\bea
&& u(p) = \half p^2, \label{eqn:u(p)} \\
&& v(p) = p. \label{eqn:v(p)}
\eea
By the Eynard-Orantin recursion \cite{Eyn-OraInv},
one can define from this curve a family of differentials:
\be
W_{g,n}(p_1, \dots, p_n) = \cW_{g,n}(p_1, \dots, p_n)  dp_1 \cdots dp_n.
\ee
Motivated by the matrix model origin of this construction,
or a definition of the Baker-Akhiezer function in \cite{Eyn-OraInv},
Gukov and Su{\l}kowski \cite{Guk-Sul} define
\be
Z = \exp \sum_{n =0}^\infty \hbar^{n-1} S_n,
\ee
where $S_n$ are defined by:
\bea
&& S_0 = \int^p v(p) du(p), \\
&& S_1(p) = - \half \log \frac{du}{dp}, \\
&& S_n(p) = \sum_{2g-1+k} \frac{(-1)^k}{k!}
\int^p \cdots \int^p W_{g,k}(p_1', \dots, p_k') dp_1'\cdots dp_k', \;\; n \geq 2.
\eea
We use a different sign convention from that in \cite{Guk-Sul}.
We will prove the following:

\begin{theorem} \label{thm:Main}
The function $Z$ satisfies the following differential equation:
\be
\hat{A}(u,v) Z = 0,
\ee
where
\be
\hat{A} = \half \hat{v}^2 - \hat{u}
\ee
is the quantization of the polynomial $A(u,v)$,
where $\hat{u} = u \cdot$, $\hat{v} = \hbar \pd_u$.
\end{theorem}

This is a special case of the general conjecture made
by Gukov-Su{\l}kowski \cite{Guk-Sul}
on quantizable algebraic curves.
We prove this result by reducing to the Dijikgraaf-Verlinde-Verlinde
recursion relation \cite{DVV}.
In a subsequent work \cite{Zho2},
we will treat the case of the local mirror curve for
$\bC^3$ and the resolved conifold.

In \cite[\S 10]{Eyn-OraRec},
Eynard and Orantin claimed that
for the curve
\bea
&& x(z) = z^2, \\
&& y(z) = z - \frac{1}{2} \sum_{k=0}^\infty t_{k+2} z^k,
\eea
the differentials $W_{g,n}(z_1,\dots, z_n)$
encodes higher Weil-Petersson volumes,
and one can also construct $F_g$ that encodes
intersection numbers $\cor{\tau_{d_1} \dots \tau_{d_n}}_g$.
See also \cite[\S 10.4.1]{Eyn-OraInv}.
In \cite[\S 10]{Eyn-OraInv},
the Eynard-Orantin recursion for the Airy curve was discussed
and the relationship between $W_{g,n}$ the Tracy-Widom kernel
was recalled.
It was also stated that ``The fact that the Baker-Akhiezer function is $Ai(x)$ and satisfies
the differential equation $Ai''=xAi$ can be seen as a consequence of the Hirota equation theorem
9.2."
The  authors of \cite{Ben-Coc-Saf-Wos}
established the equivalence of the Eynard-Orantin recursion of the Airy curve
to the DVV recursion relations via the Laplace transform of a recursion relation for
the symplectic volumes of the moduli spaces.
In \cite{Zho},
the author established
the equivalence of the DVV recursion relations to a Eynard-Orantin type recursions,
but unfortunately,
the starting point of that paper was to rewrite the DVV recursion relation using residues
and a kernel function,
not necessarily on an algebraic curve,
so the relationship with the Airy curve was missed.
After the author came across \cite{Ben-Coc-Saf-Wos} in June 2012 on the internet,
it became clear that it is possible to combine the ideas in \cite{Ben-Coc-Saf-Wos} and \cite{Zho}
to directly establish the equivalence between the DVV relations and the Eynard-Orantin
recursion for the Airy curve.
A consequence of this result and Theorem 1.1 is that
one can then relate the intersection numbers
on $\Mbar_{g,n}$ to the Airy functions $Ai(x)$ and $Bi(x)$.

The rest of the paper is arranged as follows.
In \S 2 we recall the Eynard-Orantin recursion for the Airy curve
and present a direct proof that it is equivalent to the DVV recursion.
We then change coordinate in \S 3 and combine an observation
in an earlier work \cite{Zho} to derive a simple recursion relation
for a suitably defined $n$-point polynomial functions
$$\omega_{g,n} = \sum_{a_1, \dots, a_n \geq 0}
\cor{\tau_{a_1} \cdots \tau_{a_n}}_g \prod_{i=1}^n (2a_i+1)!!w_i^{a_i+1}$$
of intersection numbers on the Deligne-Mumford moduli spaces.
It has a very simple form:
\ben
&& \omega_{g,n+1}(w_0,w_1, \dots, w_n)  \\
& = & \frac{1}{2} w_0 \omega_{g-1,n+2}(w_0,w_0, w_{[n]}) \\
& + & \frac{1}{2} w_0
\sum^s_{\substack{g_1+g_2=g \\ A_1 \coprod A_2 = [n]}}
\omega_{g_1, |A_1|+1}(w_0,w_{A_1}) \cdot \omega_{g_2, |A_2|+1}(w_0, w_{A_2}) \\
& + &  \sum_{i=1}^n D_{w_0,w_i} \omega_{g, n}(x, w_{[n]_i}),
\een
where
\ben
&& D_{u,v}x^m
= uv (u^m + 3 u^{m-1}v + 5u^{m-2}v^2 + \cdots +(2m+1)v^m).
\een
In \S 4  we consider the antiderivatives $\Omega_{g,n}$ of $\omega_{g,n}$:
\ben
\Omega_{g,n}(w_1, \dots, w_n)
= \sum_{a_1, \dots, a_n \geq 0}
\cor{\tau_{a_1} \cdots \tau_{a_n}}_g \prod_{i=1}^n (2a_i-1)!! w_i^{a_i+1/2}
\een
and integrate the above recursion relation  to get a recursion relation
\ben
&& \pd_{w_0} \Omega_{g,n+1}(w_0,w_1, \dots, w_n) \\
& = & w_0^{5/2} \pd_x\pd_y \Omega_{g-1,n+2}(x,y,w_{[n]})|_{x=y=w_0} \\
& + & w_0^{5/2} \sum^s_{\substack{g_1+g_2=g \\ A_1 \coprod A_2 = [n]}}
\pd_{w_0} \Omega_{g_1, |A_1|+1}(w_0,w_{A_1}) \cdot \pd_{w_0} \Omega_{g_2, |A_2|+1}(w_0, w_{A_2}) \\
& + & w_0^{-3/2} \sum_{i=1}^n \cD_{w_0,w_i} \pd_x \Omega_{g,n}(x, w_{[n]_i}),
\een
where $\cD_{u,v}: \bC[x] x^{-1/2} \to \bC[u,v]uv^{1/2}$
is a linear operator defined by:
\ben
&& \cD_{u,v} x^{a-1/2} = uv^{1/2}(u^{a+1} + u^av + \cdots + v^{a+1}).
\een
In \S 5 we use this result in \S 4 to
present a proof of Theorem 1.1.

\section{Eynard-Orantin Topological Recursion on Airy Curve
and Moduli Spaces of Curves}

\subsection{Eynard-Orantin recursion for the Airy curve}
\label{sec:EORecursion}

In this section,
we recall the construction of Eynard-Orantin \cite[\S 4]{Eyn-OraInv}
for the case of the Airy curve.
Near the branch point $(u,v) = 0$,
the conjugate point of $(u(p), v(p))$ is $(u(p), -v(p))$,
i.e. $\bar{p} = - p$.
Hence the vertex is given by:
\be
\omega = (v(\bar{p}) - v(p)) d u(p) = ((-p)-p) d \frac{p^2}{2}
= -2p^2dp.
\ee
Since the Airy curve has genus $0$,
under the parametrization \eqref{eqn:u(p)} and \eqref{eqn:v(p)},
the line-propagator (Bergmann kernel) on the Airy curve is given by
\cite[\S 3.2]{Eyn-OraInv}:
\be
B(p_1,p_2) = \frac{dp_1dp_2}{(p_1-p_2)^2}.
\ee
The arrow-propagator is
\be
dE_q(p) = \frac{1}{2} \int_q^{\bar{q}} B(\xi, p)
= \half dp \int_q^{-q} \frac{d\xi}{(p-\xi)^2} = \frac{qdp}{q^2-p^2}.
\ee
Hence the recursion kernel is:
\be
K(q,p) = \frac{dE_q(p)}{\omega(q)}
= \frac{dp}{q(q^2-p^2) dq}.
\ee
The Eynard-Orantin recursion has as initial values:
\bea
&& W_{0,1}(p) = 0, \label{eqn:Airy01} \\
&& W_{0,2}(p_1,p_2) = B(p_1,p_2) = \frac{dp_1dp_2}{(p_1-p_2)^2}, \label{eqn:Airy02}
\eea
and in general:
\be \label{eqn:EO1}
\begin{split}
& W_{g,n+1}(z_0,z_1, \dots, z_n)  \\
 = & \half
\res_{z=0} \biggl( K(z,z_0) \cdot
\biggl( W_{g-1,n+2}(z,-z, z_{[n]}) \\
+ & \sum_{\substack{g_1+g_2=g \\ A_1 \coprod A_2 = [n]}}
W_{g_1, |A_1|+1}(z,z_{A_1}) \cdot W_{g_2, |A_2|+1}(-z, z_{A_2})
\biggr) \biggr).
\end{split}
\ee
Here we have used the following notations:
For $A =\subset [n]$,
when $A \emptyset$, $z_A$ is empty;
otherwise, if $A=\{i_1, \dots, i_k$,
then $z_{A} = z_{i_1}, \dots, z_{i_k}$.
In terms of $\cW_{g,n}(z_1, \dots, z_n)$,
one has
\be \label{eqn:EO2}
\begin{split}
& \cW_{g,n+1}(z_0,z_1, \dots, z_n)  \\
 = & \half
\res_{z=0} \biggl( \frac{1}{z(z_0^2-z^2)} \cdot
\biggl( \cW_{g-1,n+2}(z,-z, z_{[n]}) \\
+ & \sum_{\substack{g_1+g_2=g \\ A_1 \coprod A_2 = [n]}}
\cW_{g_1, |A_1|+1}(z,z_{A_1}) \cdot \cW_{g_2, |A_2|+1}(-z, z_{A_2})
\biggr) \biggr).
\end{split}
\ee

Here are some examples obtained by applying \eqref{eqn:EO2}.
\ben
&& \cW_{0,3}(z_0,z_1,z_2) \\
& = & \res_{z=0}(\frac{1}{z(z_0^2-z^2)} \cW_{0,2}(z,z_1)\cW_{0,2}(z,z_2)) \\
& = & \res_{z=0}\biggl(\frac{1}{z(z_0^2-z^2)} \cdot \frac{1}{(z-z_2)^2}
\cdot \frac{1}{(z-z_2)^2} \biggr)\\
& = &\frac{1}{z_0^2z_1^2z_2^2}.
\een
\ben
&& \cW_{0,4}(z_0,z_1,z_2,z_3) \\
& = & \res_{z=0}(\frac{1}{z(z_0^2-z^2)}
\cdot (\cW_{0,2}(z,z_1)\cW_{0,3}(z,z_2,z_3) \\
& + & \cW_{0,2}(z,z_2)\cW_{0,3}(z,z_1,z_3)
+ \cW_{0,2}(z,z_3) \cW_{0,3}(z,z_1,z_2) ) \\
& = & \res_{z=0} \biggl( \frac{1}{z(z_0^2-z^2)} \cdot
\biggl(\frac{1}{(z-z_1)^2} \cdot \frac{1}{z^2z_2^2z_3^2} \\
& + & \frac{1}{(z-z_2)^2} \cdot \frac{1}{z^2z_1^2z_3^2}
+ \frac{1}{(z-z_3)^2} \cdot \frac{1}{z^2z_1^2z_2^2} \biggr) \biggr).
\een
Because
\ben
\res_{z=0} \frac{1}{z^3(u^2-z^2)(z-v)^2}
= \frac{3}{u^2v^4} + \frac{1}{u^4v^2},
\een
we have
\ben
\cW_{0,4}(z_0,z_1,z_2,z_3)
= \frac{1}{z_0^2z_1^2z_2^2z_3^2} \sum_{i=0}^3 \frac{3}{z_i^2}.
\een
When $(g,n)=(1,2)$ we have
\ben
&& \cW_{1,2}(z_0,z_1) \\
& = & \res_{z=0} (\frac{1}{z(z_0^2-z^2)}
(\half \cW_{0,3}(z,z,z_1) + \cW_{0,2}(z,z_1) \cW_{1,1}(z) ) ) \\
& = & \res_{z=0} \biggl(\frac{1}{z(z_0^2-z^2)} \cdot
\biggl(\half \cdot \frac{1}{z^4z_1^2} + \frac{1}{(z-z_1)^2} \cdot \frac{1}{8z^4} \biggr) \biggr) \\
& = & \half \cdot \frac{1}{z_0^6z_1^2}
+ \frac{1}{8} \biggl(\frac{5}{z_0^2z_1^6} + \frac{3}{z_0^4z_1^4} + \frac{1}{z_0^6z_1^2} \biggr) \\
& = & \frac{5}{8}\cdot \frac{1}{z_0^2z_1^6} + \frac{3}{8}\cdot \frac{1}{z_0^4z_1^4}
+ \frac{5}{8} \cdot \frac{1}{z_0^6z_1^2}.
\een

\subsection{Relationship with intersection numbers on moduli spaces of curves}
\label{sec:DVV}
Consider the intersection numbers on Deligne-Mumford moduli spaces:
\be
\cor{\tau_{a_1} \cdots \tau_{a_n}}_g: = \int_{\Mbar_{g,n}}
\psi_1^{a_1} \cdots \psi_n^{a_n}
\ee
and the following generating function
\be \label{Def:tW}
\tW_{g,n}(z_1, \dots, z_n)
= \sum_{a_1, \dots, a_n \geq 0}
\cor{\tau_{a_1} \cdots \tau_{a_n}}_g \prod_{i=1}^n \frac{(2a_i+1)!!}{z_i^{2a_i+2}}.
\ee
By evaluating the correlators by Witten-Kontsevich Theorem \cite{Wit, Kon},
one can find explicit expressions of $\tW_{g,n}$ for small $g$ and $n$.
For example,
\ben
&& \tW_{0,3}(z_1,z_2,z_3) = \frac{1}{z_1^2z_2^2z_3^2}, \label{eqn:W(0,3)} \\
&& \tW_{0,4}(z_1,z_2,z_3,z_4)
= \frac{1}{z_1^2z_2^2z_3^2z_4^2} \sum_{i=1}^4 \frac{3}{z_i^2}, \label{eqn:W(0,4)} \\
&& \tW_{0,5}(z_1,\dots,,z_5)
= \frac{1}{z_1^2\cdots z_5^2}
(\sum_{i=1}^5 \frac{15}{z_i^4}
+ \sum_{1 \leq i< j \leq 5} \frac{18}{z_i^2z_j^2} ), \label{eqn:W(0,5)} \\
&& \tW_{0,6}(z_1,\dots,,z_6)
= \frac{1}{\prod_{i=1}^6 z_i^2}
(\sum_{i=1}^6 \frac{105 }{z_i^6}
+ \sum_{1 \leq i \neq j \leq 6} \frac{135}{z_i^4z_j^2}
+ \sum_{1 \leq i < j < k \leq 6} \frac{162}{z_i^2z_j^2z_k^2}), \label{eqn:W(0,6)} \\
\een

\ben
&& \tW_{1,1}(z_1) = \frac{1}{8z_1^4}, \label{eqn:W(1,1)} \\
&& \tW_{1,2}(z_1,z_2)
= \frac{1}{8z_1^2z_2^2}
\biggl(\frac{5}{z_1^4}+ \frac{5}{z_2^4} + \frac{3}{z_1^2z_2^2} \biggr), \label{eqn:W(1,2)} \\
&& \tW_{1,3}(z_1,z_2,z_3) = \frac{1}{8z_1^2z_2^2z_3^2}\cdot
\biggl( \sum_{i=1}^3 \frac{35}{z_i^6}
+ \sum_{1  \leq i \neq j \leq 3} \frac{30}{z_i^4z_j^2}
+ \frac{18}{z_1^2z_2^2z_3^2}\biggr), \label{eqn:W(1,3)}
\een
\ben
\tW_{2,1}(z_1) & = & \frac{105}{128} \frac{1}{z_1^{10}}, \\
\tW_{2,2}(z_1,z_2) & = & \frac{1155}{128} \frac{1}{z_1^{12}z_2^2}
+ \frac{3465}{128} \frac{1}{z_1^{10}z_2^4} + \frac{6699}{128}\frac{1}{z_1^8z_2^6}
+ \frac{6699}{128}\frac{1}{z_1^6z_2^8} \\
&& + \frac{3465}{128} \frac{1}{z_1^{4}z_2^{10}} + \frac{1155}{128} \frac{1}{z_1^{2}z_2^{12}}
\een
etc.

\begin{theorem}
When $2g-2+n > 0$,
one has
\be
\cW_{g,n}(z_1, \dots, z_n) = \tW_{g,n}(z_1, \dots, z_n).
\ee
Indeed,
$\tW_{g,n}$'s satisfy the initial values and Eynard-Orantin recursion relations
\eqref{eqn:Airy01}-\eqref{eqn:EO1}.
\end{theorem}

As mentioned in the Introduction,
this result was due to \cite{Ben-Coc-Saf-Wos}.
In that work,
the result was established via an equivalent recursion relations for the symplectic
volumes of the moduli spaces \cite[Theorem 1.1]{Ben-Coc-Saf-Wos}.
Here we present a direct proof.

\begin{proof}
Recall the DVV recursion relations \cite{DVV} are:
\be \label{eqn:DVV}
\begin{split}
\cor{\tilde{\tau}_{a_0} \prod_{i=1}^n \tilde{\tau}_{a_i}}_g
= & \sum_{i=1}^n (2a_i+1) \cor{\tilde{\tau}_{a_0+a_i-1} \prod_{j \in [n]_i}
\tilde{\tau}_{a_j} }_g \\
+ & \half \sum_{b_1+b_2=a_0-2} \biggl( \cor{\tilde{\tau}_{b_1}\tilde{\tau}_{b_2} \prod_{i=1}^n \tilde{\tau}_{a_i}}_{g-1} \\
+ & \sum^s_{\substack{A_1 \coprod A_2 = [n] \\g_1+g_2=g}}
\cor{\tilde{\tau}_{b_1} \prod_{i\in A_1} \tilde{\tau}_{a_i} }_{g_1}
\cdot \cor{\tilde{\tau}_{b_2} \prod_{i\in A_2} \tilde{\tau}_{a_i} }_{g_2} \biggr),
\end{split}
\ee
where $\tilde{\tau}_a =(2a+1)!! \cdot \tau_a$ and $[n]=\{1, \dots, n\}$,
$[n]_i = [n]-\{i\}$.
In the above formula,
$$ \sum^s_{\substack{A_1 \coprod A_2 = [n] \\g_1+g_2=g}}$$
means the summation is taken over the ``stable cases",
i.e.:
$$2g_1-1+|A_1| >0, \;\; 2g_2-1+|A_2|>0.$$
Multiply both sides of \eqref{eqn:DVV} by $\frac{1}{z_0^{2a_0+2}} \cdot \prod_{i=1}^n \frac{1}{z_i^{2a_i+2}}$
and take summations over $a_0, a_1, \dots, a_n$:
\ben
&& \tW_{g,n+1}(z_0,z_1, \dots, z_n) \\
& = & \sum_{a_0,a_1, \cdots, a_n \geq 0} \cor{\tau_{a_0}\tau_{a_1} \cdots \tau_{a_n}}_g
\frac{(2a_0+1)!!}{z_0^{2a_0+2}} \cdot \prod_{i=1}^n \frac{(2a_i+1)!!}{z_i^{2a_i+2}} \\
& = & \sum_{a_0,a_1, \cdots, a_n \geq 0} \sum_{i=1}^n (2a_i+1) \cdot (2a_0+2a_i-1)!! \\
&& \cdot \cor{\tau_{a_0+a_i-1} \prod_{j\in [n]_i} \tau_{a_j} }_g
\cdot \frac{1}{z_0^{2a_0+2}} \cdot \frac{1}{z_i^{2a_i+2}} \cdot \prod_{j \in [n]_i} \frac{(2a_j+1)!!}{z_j^{2a_j+2}} \\
& + & \half \sum_{b_1+ b_2= a_0-2} (2b_1+1)!!\cdot (2b_2+1)!!
\biggl( \cor{\tau_{b_1}\tau_{b_2} \prod_{i=1}^n \tau_{a_i}}_{g-1} \\
& + &  \sum^s_{\substack{g_1+g_2=g\\A_1 \coprod A_2}}
\cor{\tau_{b_1} \prod_{i\in A_1} \tau_{a_i} }_{g_1}
\cdot  \cor{\tau_{b_2} \prod_{i\in A_2} \tau_{a_i} }_{g_2} \biggr) \\
&&  \cdot
\frac{1}{z_0^{2b_1+2b_2+6}}\cdot \prod_{i=1}^n \frac{(2a_i+1)!!}{z_i^{2a_i+2}}.
\een
Notice that
\be
\res_{z=0} \biggl(\frac{1}{z(z_0^2-z^2)} \cdot \frac{1}{z^{2b_1+2}} \cdot \frac{1}{z^{2b_2+2}} \biggr)
= \frac{1}{z_0^{2b_1+2b_2+6}},
\ee
and
\be
\res_{z=0} \biggl( \frac{1}{z(z_0^2-z^2)} \cdot  \frac{1}{(z-z_i)^2} \cdot \frac{1}{z^{2m+2}} \biggr)
= \sum_{a_0 =0}^{m+1-a_0} \frac{(2a_i+1)}{z_0^{2a_0+2}z_i^{2a_i+2}},
\ee
so the above equality an be rewritten as follows:
\ben
&& \tW_{g,n+1}(z_0, z_1, \dots, z_n) \\
& = & \sum_{i=1}^n \res_{z=0} \biggl( \frac{1}{z(z_0-z^2)} \cdot \biggl( \frac{1}{(z-z_i)^2}
\cdot \tW_{g,n}(-z, [n]_i) \\
& + & \half \tW_{g-1,n+2}(z,-z, [n]) \\
& + &  \sum^s_{\substack{g_1+g_2=g\\A_1 \coprod A_2}}
\tW_{g_1,|A_1|+1}(z, z_{A_1}) \cdot \tW_{g_2, |A_2|+!}(-z, z_{A_2}) \biggr) \biggr) .
\een
The proof is completed by setting:
$$\tW_{0,2}(z_1,z_2) = \frac{1}{(z_1-z_2)^2}.$$
\end{proof}

\section{Polynomial Reformulations}

In this section we will show that in different coordinates
the recursion relations in the preceding section can be
reformulated as operations on polynomials.

\subsection{Change of variable}
In \S \ref{sec:EORecursion} and \ref{sec:DVV} we have presented some examples
of $\cW_{g,n}(z_1, \dots, z_n)$.
From these examples,
it is clear that after the following change of variables
$$w_i = \frac{1}{z_i^2},$$
One gets polynomial expressions:
\be
\omega_{g,n}(w_1, \dots, w_n)
= \cW_{g,n}(z_1, \dots, z_n).
\ee
By \eqref{Def:tW},
\be \label{Def:omega}
\omega_{g,n}(w_1, \dots, w_n)
= \sum_{a_1, \dots, a_n \geq 0}
\cor{\tau_{a_1} \cdots \tau_{a_n}}_g \prod_{i=1}^n (2a_i+1)!!w_i^{a_i+1}.
\ee
The following are some examples.
\ben
&& \omega_{0,3}(w_1,w_2,w_3) = w_1w_2w_3,  \\
&& \omega_{0,4}(w_1, \dots, w_4) = w_1\cdots w_4 \sum_{i=1}^4 w_i, \\
&& \omega_{0,5}(w_1, \dots, w_5)
= w_1\cdots w_5 (15\sum_{i=1}^5 w_i^2+18  \sum_{1 \leq i < j \leq 5} w_i w_j), \\
&& \omega_{0,6}(w_1,\dots,w_6)
= \prod_{i=1}^6 w_i^2\cdot
(105\sum_{i=1}^6  w_i^3
+  135 \sum_{1 \leq i \neq j \leq 6} w_i^2w_j \\
&& \;\;\;\; \;\;\;\;\;\; + 162 \sum_{1 \leq i < j < k \leq 6} w_iw_jw_k ),
\een
\ben
&& \omega_{1,1}(w_1) = \frac{1}{8}w_1^2,   \\
&& \omega_{1,2}(w_1,w_2)
= \frac{w_1w_2}{8} (5w_1^2+5w_2^2 + 3w_1w_2 ),  \\
&& \omega_{1,3}(w_1,w_2,w_3)  = \frac{w_1w_2w_3}{8} \cdot
( \sum_{i=1}^3 35w_i^3 + \sum_{1  \leq i \neq j \leq 3} 30w_i^2w_j
+ 18 w_1w_2w_3 \biggr),
\een
\ben
\omega_{2,1}(w_1) & = &  \frac{105}{128} w_1^5, \\
\omega_{2,2}(w_1,w_2) & = & \frac{w_1w_2}{128}\bigl( 1155(w_1^5+w_2^5)
+ 3465 (w_1^4w_2+w_1w_2^4) \\
& + & 6699(w_1^3w_2^2+w_1^2w_2^3) \bigr).
\een

\subsection{Recursion relations for $\omega_{g,n}$}

\begin{theorem} \label{them:omega}
Except for the case of $(g,n) = (0,2)$,
the following recursion relations hold:
\be \label{eqn:omegaRec}
\begin{split}
& \omega_{g,n+1}(w_0,w_1, \dots, w_n)  \\
= & \frac{1}{2} w_0 \omega_{g-1,n+2}(w_0,w_0, w_{[n]}) \\
+ & \frac{1}{2} w_0
\sum^s_{\substack{g_1+g_2=g \\ A_1 \coprod A_2 = [n]}}
\omega_{g_1, |A_1|+1}(w_0,w_{A_1}) \cdot \omega_{g_2, |A_2|+1}(w_0, w_{A_2}) \\
+ &  \sum_{i=1}^n D_{w_0,w_i} \omega_{g, n}(x, w_{[n]_i}),
\end{split}
\ee
where for $m \geq 0$,
\be
D_{u,v} x^m = uv(u^m + 3u^{m-1}v + 5 u^{m-2}v^2 + \cdots +(2m+1) v^m).
\ee
\end{theorem}

\begin{proof}
By \eqref{Def:omega},
\eqref{eqn:omegaRec} is equivalent to \eqref{eqn:DVV}.
\end{proof}

The $(g,n) = (0,2)$ case is exceptional and can be treated separately as follows.
\ben
&& \omega_{0,3}(w_0,w_1, w_2)  \\
& = &  \int_{|u|=1/\epsilon^{1/2}}
\frac{w_0}{u-w_0} \cdot  \omega_{0, 2}(u,w_1) \cdot \omega_{0, 2}(u, w_2)  du \\
& = & \int_{|u|=1/\epsilon^{1/2}}
\frac{w_0}{u-w_0}
\cdot \frac{1}{(u-w_1)^2}
\cdot \frac{1}{(u-w_2)^2}  du.
\een
A complicated residue calculation by Maple yields:
\ben
\omega_{0,3}(w_0,w_1,w_2) = w_0w_1w_2.
\een
This is a match with \eqref{eqn:W(0,3)}.

\subsection{Examples}

\subsubsection{The $(g,n) = (0,3)$ case}
This is the first case of \eqref{eqn:omegaRec}:
\ben
&& \omega_{0,4}(w_0,w_1,w_2, w_3)
=  \sum_{i=1}^3 D_{w_0,w_i} \omega_{0, 3}(x, w_{[3]_i}) \\
& = & \sum_{i=1}^3 D_{w_0,w_i} (x \cdot \frac{w_1w_2w_3}{w_i})
= \sum_{i=1}^3 w_0w_i(w_0+3w_i) \cdot \frac{w_1w_2w_3}{w_i} \\
& = & 3w_0^2w_1w_2w_3 + 3w_0w_1w_2w_3 \sum_{i=1}^3 w_i
= 3 w_0w_1w_2w_3 \sum_{i=0}^3 w_i.
\een

\subsubsection{The $(g,n) = (1,1)$ case}

\ben
\omega_{1,2}(w_0,w_1)
& = & \frac{1}{2} w_0 \omega_{0,3}(w_0,w_0, w_1)
+  D_{w_0,w_1} \omega_{1, 1}(x) \\
& = & \half w_0\cdot w_0^2w_1
+ D_{w_0,w_1} \frac{1}{8} x^2 \\
& = & \frac{1}{2} w_0^3w_1
+ \frac{1}{8} w_0w_1(w_0^2+3w_0w_1 + 5 w_1^2) \\
& = & \frac{1}{8} w_0w_1(5w_0^2+5w_1^2 + 3w_0w_1).
\een

\subsection{Derivation of \eqref{eqn:omegaRec}}

Let us explain how \eqref{eqn:omegaRec} was derived originally.
In \cite{Zho},
the author has shown that the DVV recursion relations
are equivalent to Eynard-Orantin type recursion with
initial value
$$\cW_{0,2}(z_1,z_2) = \frac{z_1^2+z_2^2}{(z_1^2-z_2^2)^2}$$
and kernel function:
$$
K(z_1,z_2) = \frac{1}{z_1z_2(z_2-z_1)}.
$$
It is easy to check that the same holds for
\ben
&& \cW_2(z_1,z_2) = \frac{z_1^2+z_2^2}{(z_1^2-z_2^2)^2}, \\
&& K(z_1,z_2) = \frac{1}{z_1(z_2^2-z_1^2)}.
\een
In the $w$ coordinates,
we have
\ben
\omega_2(w_1,w_2) = \frac{w_1w_2(w_1+w_2)}{(w_1-w_2)^2}
= w_2 + \frac{3w_2^2}{w_1-w_2} +\frac{2w_2^3}{(w_1-w_2)^2}.
\een
and we have
\ben
&& \omega_{g,n+1}(w_0,w_1, \dots, w_n)  \\
& = & \half \cdot \frac{1}{2\pi i}
\int_{|z|=\epsilon} \biggl( \frac{1}{z(\frac{1}{w_0}-z^2)} \cdot
\biggl( \omega_{g-1,n+2}(\frac{1}{z^2},\frac{1}{z^2}, w_{[n]}) \\
& + & \sum_{\substack{g_1+g_2=g \\ A_1 \coprod A_2 = [n]}}
\omega_{g_1, |A_1|+1}(\frac{1}{z^2},w_{A_1}) \cdot \omega_{g_2, |A_2|+1}(\frac{1}{z^2}, w_{A_2})
\biggr) \biggr)dz \\
& = &-\frac{1}{2\pi i}\int_{|w|=1/\epsilon^{1/2}} \frac{w^{1/2}}{\frac{1}{w_0}-\frac{1}{w}}
\cdot \biggl( \omega_{g-1,n+2}(w,w, w_{[n]}) \\
& + & \sum_{\substack{g_1+g_2=g \\ A_1 \coprod A_2 = [n]}}
\omega_{g_1, |A_1|+1}(w,w_{A_1}) \cdot \omega_{g_2, |A_2|+1}(w, w_{A_2})
\biggr) d\frac{1}{w^{1/2}} \\
& = & \frac{1}{2}\cdot \frac{1}{2\pi i}\int_{|w|=1/\epsilon^{1/2}}
\frac{w_0}{w-w_0}
\cdot \biggl( \omega_{g-1,n+2}(w,w, w_{[n]}) \\
& + & 2 \sum_{i=1}^n \omega_{0, 2}(w,w_i) \cdot \omega_{g, n}(w, w_{[n]_i}) \\
& + & \sum^s_{\substack{g_1+g_2=g \\ A_1 \coprod A_2 = [n]}}
\omega_{g_1, |A_1|+1}(w,w_{A_1}) \cdot \omega_{g_2, |A_2|+1}(w, w_{A_2})
\biggr) dw.
\een
The first line on the right-hand side of the last equality is:
\be
I = \frac{1}{2} w_0 \omega_{g-1,n+2}(w_0,w_0, w_{[n]}).
\ee
The third line on the
right-hand side of the last equality is:
\be
III = \frac{1}{2} w_0
\sum^s_{\substack{g_1+g_2=g \\ A_1 \coprod A_2 = [n]}}
\omega_{g_1, |A_1|+1}(w_0,w_{A_1}) \cdot \omega_{g_2, |A_2|+1}(w_0, w_{A_2}).
\ee
The second line on the
right-hand side of the last equality is more complicated.
Recall that if $f(z)$ is holomorphic at $z=z_0$,
then one has
\ben
&& \res_{z=z_0} \frac{f(z)}{z-z_0} = f(z_0), \\
&& \res_{z=z_0} \frac{f(z)}{(z-z_0)^2} = f'(z_0).
\een
So we have
\ben
II & = & \frac{1}{2\pi i}\int_{|w|=\epsilon^{-1/2}}
\frac{w_0}{u-w_0}\cdot  \sum_{i=1}^n \omega_{0, 2}(u,w_i) \cdot \omega_{g, n}(u, w_{[n]_i}) du \\
& = & \frac{1}{2\pi i} \int_{|w|=\epsilon^{-1/2}}
\frac{w_0}{u-w_0} \cdot  \sum_{i=1}^n\biggl(
w_i+ \frac{3w_i^2}{u-w_i} + \frac{2w_i^3}{(u-w_i)^2} \biggr) \cdot \omega_{g, n}(u, w_{[n]_i}) du \\
& = &  \sum_{i=1}^n w_0 \cdot \frac{w_0w_i(w_0+w_i)}{(w_0-w_i)^2}
\cdot \omega_{g, n}(w_0, w_{[n]_i})
+  \sum_{i=1}^n \frac{w_0}{w_i-w_0}\cdot 3w_i^2 \cdot \omega_{g,n}(w_{[n]}) \\
& - & \sum_{i=1}^n \frac{w_0}{(w_i-w_0)^2} \cdot 2\omega_i^3 \cdot \omega_{g, n}(w_{[n]})
+ \sum_{i=1}^n \frac{w_0}{w_i-w_0} \cdot 2\omega_i^3 \pd_{w_i} \omega_{g, n}(w_{[n]}).
\een
Now we introduce a linear operator
\be
D_{u,v}: \bC[x] \to \bC[u,v],
\ee
defined as follows:
\ben
D_{u,v} f(x):
= uv \cdot \biggl( \frac{u(u+v)}{(u-v)^2} \cdot f(u)
+ \frac{3v}{v-u} f(v)
- \frac{2v^2}{(v-u)^2} f(v)
+ \frac{2v^2}{v-u} \cdot f'(v)\biggr).
\een
With this operator,
we have:
\ben
&& \omega_{g,n+1}(w_0,w_1, \dots, w_n)  \\
& = & \frac{1}{2} w_0 \omega_{g-1,n+2}(w_0,w_0, w_{[n]}) \\
& + & \frac{1}{2} w_0
\sum^s_{\substack{g_1+g_2=g \\ A_1 \coprod A_2 = [n]}}
\omega_{g_1, |A_1|+1}(w_0,w_{A_1}) \cdot \omega_{g_2, |A_2|+1}(w_0, w_{A_2}) \\
& + &  \sum_{i=1}^n D_{w_0,w_i} \omega_{g, n}(x, w_{[n]_i}).
\een

\begin{lemma}
For $m \geq 0$, one has
\ben
&& D_{u,v}x^m
= uv (u^m + 3 u^{m-1}v + 5u^{m-2}v^2 + \cdots +(2m+1)v^m).
\een
\end{lemma}

\begin{proof}
This is elementary:
\ben
&& \frac{u(u+v)}{(u-v)^2} \cdot u^m
+ \frac{3v}{v-u} \cdot v^m
- \frac{2v^2}{(v-u)^2} \cdot v^m
+ \frac{2v^2}{v-u} \cdot mv^{m-1} \\
& = & \frac{u^{m+2}+u^{m+1}v-2v^{m+2}}{(u-v)^2}
- \frac{(2m+3)v^{m+1}}{u-v} \\
& = &  \frac{u^{m+1}(u-v)+2(u^{m+1}-v^{m+1})v}{(u-v)^2}
- \frac{(2m+3)v^{m+1}}{u-v} \\
& = & \frac{1}{u-v}(u^{m+1} + 2(u^m+u^{m-1}v+ \cdots + v^{m-1})v
- (2m+3) v^{m+1}) \\
& = & \frac{1}{u-v}[(u^{m+1} - v^{m+1})
+ 2(u^m-v^m)v+2(u^{m-1}-v^{m-1})v^2 \\
& + & \cdots + (u-v)v^m ] \\
& = & u^m + 3 u^{m-1}v + 5u^{m-2}v^2 + \cdots +(2m+1)v^m.
\een
\end{proof}

\section{Recursion Relations for the Antiderivatives of $\omega_{g,n}$}

\subsection{The antiderivatives of $\omega_{g,n}$}
These are defined by:
\be \label{Def:Omega}
\Omega_{g,n}
= \sum_{a_1, \dots, a_n \geq 0}
\cor{\tau_{a_1} \cdots \tau_{a_n}}_g \prod_{i=1}^n (2a_i-1)!! w_i^{a_i+1/2}.
\ee
Here we use the following convention:
\be
(-1)!! = 1.
\ee
For example,
\ben
&& \Omega_{0,3}(w_1,w_2,w_3)= (w_1 w_2 w_3)^{1/2}, \\
&& \Omega_{0,4}(w_1,\dots, w_4) = w_1^{1/2} \cdots w_4^{1/2} \sum_{i=1}^4 w_i, \\
&& \Omega_{0,5}(w_1,\dots, w_4) =  (w_1\cdots w_4)^{1/2}
\biggl( \sum_{i=1}^4 3w_i^2 + \sum_{1 \leq i < j \leq 5} 2w_iw_j \biggr),
\een
\ben
&& \Omega_{1,1}(w_1) = \frac{1}{24}w_1^{3/2},   \\
&& \Omega_{1,2}(w_1,w_2)
= \frac{(w_1w_2)^{1/2}}{24} \biggl(3w_1^2+3w_2^2 +w_1w_2 \biggr),  \\
&& \Omega_{1,3}(w_1,w_2,w_3) = \frac{1}{24(\omega_1\omega_2\omega_3)^{1/2}}\cdot
\biggl( 15\sum_{i=1}^3 \omega_i^3
+ 6 \sum_{1  \leq i \neq j \leq 3} w_i^2w_j + 2w_1w_2w_3 \biggr).
\een

\begin{lemma}
The functions $\omega_{g,n}$ and $\Omega_{g,n}$ are related by:
\be
\omega_{g,n}
= 2^n \prod_{j=1}^n w_j^{3/2}
\cdot \pd_{w_1} \cdots \pd_{w_n} \Omega_{g,n}.
\ee
\end{lemma}

\begin{proof}
This is easy to see from \eqref{Def:omega} and \eqref{Def:Omega}.
\end{proof}

\subsection{Recursion relations for $\Omega_{g,n}$}

We need the following easy observation.

\begin{lemma}
The following identity holds:
\be
\begin{split}
& D_{w_0,w_i} \omega_{g, n}(x, w_{[n]_i}) \\
= & 2^{n+1}   (w_1\dots w_n)^{3/2} \pd_{w_1}\cdots \pd_{w_n} \cD_{w_0,w_i}
\pd_x \Omega_{g,n}(x, w_{[n]_i}),
\end{split}
\ee
where $\cD_{u,v}: \bC[x] x^{-1/2} \to \bC[u,v]uv^{1/2}$
is a linear operator defined by:
\be
\cD_{u,v} x^{a-1/2} = uv^{1/2}(u^{a+1} + u^av + \cdots + v^{a+1}).
\ee
\end{lemma}

\begin{proof}
By \eqref{Def:omega},
\ben
&& \omega_{g,n}(x,w_{[n]_i}) \\
& = & \sum_{a, a_1, \dots, \hat{a_i}, \dots, a_n \geq 0}
\cor{\tau_a \tau_{a_1} \cdots \widehat{\tau_{a_i}} \cdots \tau_{a_n}}_g
(2a+1)!!x^{a+1}
\prod_{\substack{1\leq j \leq n\\ j\neq i}} (2a_j+1)!!w_j^{a_j+1},
\een
and so
\ben
&& D_{w_0,w_i} \omega_{g,n}(x,w_{[n]_i}) \\
& = & \sum_{a, a_1, \dots, \hat{a_i}, \dots, a_n \geq 0}
\cor{\tau_a \tau_{a_1} \cdots \widehat{\tau_{a_i}} \cdots \tau_{a_n}}_g
(2a+1)!! \cdot D_{w_0,w_i}x^{a+1} \\
&& \cdot
\prod_{\substack{1\leq j \leq n\\ j\neq i}} (2a_j+1)!!w_j^{a_j+1} \\
& = & \sum_{a, a_1, \dots, \hat{a_i}, \dots, a_n \geq 0}
\cor{\tau_a \tau_{a_1} \cdots \widehat{\tau_{a_i}} \cdots \tau_{a_n}}_g
 \cdot
\prod_{\substack{1\leq j \leq n\\ j\neq i}} (2a_j+1)!!w_j^{a_j+1} \\
&& \cdot (2a+1)!! \cdot w_0w_i (w_0^{a+1}+3w_0^aw_i + \cdots + (2a+3) w_i^{a+1})  \\
& = & 2^{n+1} (w_1\dots w_n)^{3/2} \pd_{w_1}\cdots \pd_{w_n}
\sum_{a, a_1, \dots, \hat{a_i}, \dots, a_n \geq 0}
\cor{\tau_a \tau_{a_1} \cdots \widehat{\tau_{a_i}} \cdots \tau_{a_n}}_g \\
&& \cdot
\prod_{\substack{1\leq j \leq n\\ j\neq i}} (2a_j-1)!!w_j^{a_j+1/2} \cdot (2a+1)!! \\
&& \cdot w_0w_i^{1/2} (w_0^{a+1}+w_0^aw_i + \cdots + w_i^{a+1})  \\
& = & 2^{n+1} (w_1\dots w_n)^{3/2} \pd_{w_1}\cdots \pd_{w_n} \cD_{w_0,w_i}
\pd_x \Omega_{g,n}(x, w_{[n]_i}).
\een
\end{proof}

With the above two Lemmas,
it is easy to derive from Theorem

\begin{theorem} \label{thm:Omega}
Except for the case of $(g,n) = (0,2)$,
the following recursion relations hold:
\be \label{eqn:OmegaRec}
\begin{split}
& \pd_{w_0} \Omega_{g,n+1}(w_0,w_1, \dots, w_n) \\
= & w_0^{5/2} \pd_x\pd_y \Omega_{g-1,n+2}(x,y,w_{[n]})|_{x=y=w_0} \\
+ & w_0^{5/2} \sum^s_{\substack{g_1+g_2=g \\ A_1 \coprod A_2 = [n]}}
\pd_{w_0} \Omega_{g_1, |A_1|+1}(w_0,w_{A_1}) \cdot \pd_{w_0} \Omega_{g_2, |A_2|+1}(w_0, w_{A_2}) \\
+ & w_0^{-3/2} \sum_{i=1}^n \cD_{w_0,w_i} \pd_x \Omega_{g,n}(x, w_{[n]_i}).
\end{split}
\ee
\end{theorem}

\subsection{Examples}
\subsubsection{The $(g,n) = (0,3)$ case}
\ben
\pd_{w_0} \Omega_{0,4}(w_0,w_1,w_2,w_3)
& = & w_0^{-3/2} \sum_{i=1}^3 \cD_{w_0,w_i}
\pd_x \Omega_{0, 3}(x, w_{[3]_i}) \\
& = &  w_0^{-3/2} \sum_{i=1}^3 \cD_{w_0,w_i} \pd_x (x^{1/2} \cdot
\frac{(w_1w_2w_3)^{1/2}}{w_i^{1/2}}) \\
& = & \half w_0^{-3/2}\sum_{i=1}^3 w_0w_i^{1/2} (w_0+w_i) \cdot \frac{(w_1w_2w_3)^{1/2}}{w_i^{1/2}} \\
& = & \frac{3}{2}w_0^{1/2} (w_1w_2w_3)^{1/2}
+ \half w_0^{-1/2} (w_1w_2w_3)^{1/2} \sum_{i=1}^3 w_i.
\een

\subsubsection{The $(g,n) = (1,1)$ case}

\ben
\pd_{w_0} \Omega_{1,2}(w_0,w_1)
& = & w_0^{5/2} \pd_{x}\pd_y  \Omega_{0,3}(x,y, w_1)|_{x=y=w_0}
+ u^{-3/2}  \cD_{w_0,w_1} \Omega_{1, 1}(x) \\
& = & w_0^{5/2} \cdot \pd_x\pd_y((xyw_1)^{1/2})|_{x=y=w_0}
+ \cD_{w_0,w_1} \pd_x \frac{1}{24} x^{3/2} \\
& = & \frac{1}{4} w_0^{3/2}w_1^{1/2}
+ \frac{1}{16} w_0^{-3/2} \cdot w_0w_1^{1/2}(w_0^2+w_0w_1 + w_1^2) \\
& = & \frac{1}{16} w_0^{-1/2}w_1^{1/2}(5w_0^2 + w_0w_1 +w_1^2).
\een

\section{Gukov-Su{\l}kowski Conjecture for the Airy Curve}

\subsection{The formulation of the conjecture in $u$-coordinates}

Quantization of the defining polynomial of the Airy curve
\be
A = \frac{1}{2} v^2 - u
\ee
by the assignment
\be
\hat{u} = u, \;\;\; \hat{v} = \hbar \pd_u.
\ee
yields the following differential operator:
\be
\hat{A} = \frac{1}{2} \hbar^2 \pd_u^2 - u.
\ee
Under the following parametrization
\bea
&& u(z) = \frac{1}{2} z^2, \\
&& v(z) = z,
\eea
we have
\be
S_0 = \int^z v(z) du(z) = \int^z z^2dz = \frac{1}{3}z^3,
\ee
\be
S_1 = - \frac{1}{2} \log \frac{du}{dz} = - \half \log (z),
\ee
and
\be
\begin{split}
S_n = & \sum_{2g-1+k=n} \frac{(-1)^k}{k!} \int^z dz_1'\cdots \int^z dz_n'
\cW_{g,k}(z_1', \dots, z_k') \\
= & \sum_{2g-1+k=n} \frac{(-1)^k}{k!} \Xi_{g,k}(z, \dots, z),
\end{split}
\ee
where
\bea
&& \Xi_{g,n}(z_1,\dots,z_n)
= \int^{z_1} \cdots \int^{z_n} \cW_{g,n}(z_1, \dots, z_n)
dz_1 \cdots dz_n \\
& = & (-1)^n  \sum_{a_1, \dots, a_n \geq 0}
\cor{\tau_{a_1} \cdots \tau_{a_n}}_g \prod_{i=1}^n
\frac{(2a_i-1)!!}{z_i^{2a_i+1}}.
\eea
For example,
\ben
S_2 & = &  -\Xi_{1,1}(z) - \frac{1}{3!} \Xi_{0,3}(z,z,z)
= \frac{1}{24z^3}  + \frac{1}{6z^3}  = \frac{5}{24z^3}.
\een
\ben
S_3
= \frac{1}{2!} \Xi_{1,2}(z,z)
+ \frac{1}{4!} \Xi_{0,4}(z,z,z,z)
= \frac{1}{2} \cdot \frac{7}{24z^6} + \frac{1}{24} \cdot \frac{4}{z^6}
= \frac{5}{16 z^6}.
\een
Choose $z= u^{1/2}$ or $z = - u^{1/2}$,
one then expresses $S_n$ in the $u$-coordinates.
For example,
\bea
&& S_0 =\pm  \frac{1}{3} (2u)^{3/2},  \label{eqn:S0} \\
&& S_1 = - \frac{1}{4} \log (2u) + constant,  \label{eqn:S1}\\
&& S_2 = \pm \frac{5}{24(2u)^{3/2}},  \label{eqn:S2} \\
&& S_3 = \frac{5}{16} \frac{1}{(2u)^3}.  \label{eqn:S3}
\eea

Recall
\be
Z = \exp \sum_{n=0}^\infty \hbar^{n-1} S_n.
\ee
Therefore,
\ben
\hat{A}Z & = & (\half \hbar^2\pd_u^2 - u) \exp \sum_{n=0}^\infty \hbar^{n-1} S_n \\
& = & \half \hbar \pd_u \biggl( \sum_{n=0}^\infty \hbar^n \pd_u S_n \cdot
\exp \sum_{n=0}^\infty \hbar^{n-1} S_n \biggr)
- u \cdot\exp \sum_{n=0}^\infty \hbar^{n-1} S_n \\
& = &  \biggl( \half  \biggl( \sum_{n=0}^\infty \hbar^n \pd_u S_n\biggr)^2
+ \half \sum_{n=0}^\infty \hbar^{n+1} \pd_u^2 S_n
- u \biggr) \cdot\exp \sum_{n=0}^\infty \hbar^{n-1} S_n \\
& = & \biggl((\half (\pd_uS_0)^2- u) + \sum_{n=1}^\infty \hbar^n (\half \pd_u^2 S_{n-1}
+ \half \sum_{i+j=n} \pd_uS_i \cdot \pd_u S_j) \biggr) \cdot Z.
\een
It follows that the equation
\be
\hat{A} Z = 0
\ee
is equivalent to the following sequence of equations:
\bea
&& \half (\pd_u S_0)^2 = u, \label{eqn:Order0} \\
&& \half \pd_u^2S_0 + \pd_u S_0 \cdot \pd_u S_1 = 0,  \label{eqn:Order1} \\
&& \half \pd_u^2 S_1 + \pd_u S_0 \cdot \pd_u S_2 + \half \pd_u S_1 \cdot \pd_u S_1 = 0,
\label{eqn:Order2} \\
&& \half \pd_u^2 S_{n-1} + \pd_u S_0 \cdot \pd_u S_n
+ \pd_u S_1 \cdot \pd_u S_{n-1}
+ \half \sum_{\substack{i+j=n\\i,j\geq 2}} \pd_u S_i \cdot \pd_u S_j = 0,
\label{eqn:OrderN>2}
\eea
where $n > 2$.
One can plug in \eqref{eqn:S0}-\eqref{eqn:S2} to check that
\eqref{eqn:Order0}-\eqref{eqn:Order2} hold
and one can rewrite \eqref{eqn:OrderN>2} as follows:
\be \label{eqn:eqn:OrderN>=3}
\pd_u S_n = \pm \frac{1}{2(2u)^{1/2}} \biggl(- \pd_u^2 S_{n-1}
+ \frac{1}{2u} \pd_u S_{n-1}
- \sum_{\substack{i+j=n\\i,j \geq 2}} \pd_u S_i \cdot \pd_u S_j \biggr).
\ee
It suffices to prove this for $n \geq 3$.

\subsection{The reformulation in $w$-coordinates}

Because
\ben
&& w(z) = \frac{1}{z^2},
\een
we have
\ben
&& w = \frac{1}{2u}, \\
&& \pd_u = - 2 w^2 \pd_w, \\
&& \pd_u^2 = 4 w^4 \pd_w^2 + 8 w^3 \pd_w.
\een
Hence one can rewrite \eqref{eqn:eqn:OrderN>=3}
\be \label{eqn:eqn:OrderN>=3}
\pd_u S_n = \pm \frac{1}{2(2u)^{1/2}} \biggl(- \pd_u^2 S_{n-1}
+ \frac{1}{2u} \pd_u S_{n-1}
- \sum_{\substack{i+j=n\\i,j \geq 2}} \pd_u S_i \cdot \pd_u S_j \biggr).
\ee
as follows:
\ben
&&  - 2 w^2 \pd_w S_n = \pm \frac{w^{1/2}}{2}
\biggl(- (4 w^4 \pd_w^2 + 8 w^3 \pd_w) S_{n-1}\\
& - &  \frac{w}{2} \cdot 2w^2\pd_w S_{n-1}
- \sum_{\substack{i+j=n\\i,j \geq 2}} 2w^2 \pd_wS_i \cdot
2w^2\pd_w S_j \biggr).
\een
After simplification one gets
\be
\pd_w S_n = \pm\biggl( \frac{5}{2}w^{3/2}\pd_w S_{n-1}
+ w^{5/2}\cdot \pd_w^2S_{n-1}
+ w^{5/2} \sum_{\substack{i+j=n\\i,j\geq 2}} \pd_w S_i \cdot \pd_wS_j \biggr),
\ee
or equivalently,
\be \label{eqn:PartialS}
w^{5/2} \pd_w S_n
= \pm \biggl((w^{5/2}\pd_w)^2  S_{n-1}
+ \sum_{\substack{i+j=n\\i,j\geq 2}} w^{5/2} \pd_w S_i \cdot w^{5/2} \pd_wS_j
\biggr).
\ee

\begin{remark}
Let $t = -\frac{2}{3} w^{-3/2}$,
then one has
\be
\pd_w = \frac{\pd t}{\pd w} \pd_t = w^{-5/2} \pd_t,
\ee
and so
\be
\pd_t S_n = \pd_t^2 S_{n-1}
+ \sum_{\substack{i+j=n \\i, j \geq 2}} \pd_t S_i \cdot \pd_t S_j.
\ee
\end{remark}

\subsection{The proof of Theorem }
It suffices now to establish \eqref{eqn:PartialS}.
By comparing with \eqref{Def:Omega},
we have
\be
\Omega_{g,n}(w_1, \dots, w_n) = (\mp 1)^n \Xi_{g,n}(z_1, \dots ,z_n),
\ee
and so
\be \label{eqn:SinOmega}
\begin{split}
S_n = &
\sum_{2g-1+k=n} \frac{(-1)^k(\mp1)^k}{k!} \Omega_{g,k}(w, \dots, w) \\
= & (\pm 1)^{n+1}  \sum_{2g-1+k=n} \frac{1}{k!} \Omega_{g,k}(w, \dots, w).
\end{split}
\ee

By \eqref{eqn:SinOmega},
\ben
w^{5/2} \pd_w S_n
= (\pm 1)^{n+1} \sum_{2g-1+k=n} \frac{1}{(k-1)!} w^{5/2}\pd_w \Omega_{g,k}(w, \dots, w),
\een
and
\ben
&& (w^{5/2}\pd_w)^2 S_n
= (\pm1)^{n+1}\sum_{2g-1+k=n} \frac{1}{(k-1)!} (x^{5/2}\pd_x)^2 \Omega_{g,k}(x,w, \dots, w)|_{x=w} \\
& + & (\pm1)^{n+1}\sum_{2g-1+k=n} \frac{1}{(k-2)!} x^{5/2}\pd_x
y^{5/2}\pd_y \Omega_{g,k}(x,y,w, \dots, w)|_{x=y=w}.
\een
Hence \eqref{eqn:PartialS} can be rewritten as follows
(after taking care of the $\pm$ signs):
\be \label{eqn:MainInW}
\begin{split}
& \sum_{2g-1+k=n} \frac{1}{(k-1)!} w^{5/2}\pd_w \Omega_{g,k}(w, \dots, w) \\
= &
 \sum_{2g-1+k=n-1} \frac{1}{(k-1)!} (x^{5/2}\pd_x)^2 \Omega_{g,k}(x,w, \dots, w)|_{x=w} \\
 + & \sum_{2g-1+k=n-1} \frac{1}{(k-2)!} x^{5/2}\pd_x y^{5/2}\pd_y \Omega_{g,k}(x,y,w, \dots, w)|_{x=y=w}
\biggr) \\
 + & w^{5/2} \sum_{\substack{i+j=n\\i,j\geq 2}}\sum_{2g_1-1+k_1=i}
\frac{1}{(k_1-1)!}w^{5/2} \pd_w \Omega_{g_1,k_1}(w, \dots, w)\\
& \cdot \sum_{2g_2-1+k_2=j} \frac{1}{(k_2-1)!}
w^{5/2}\pd_w \Omega_{g_2,k_2}(w, \dots, w).
\end{split}
\ee
We now show that this can be derived from \eqref{eqn:OmegaRec}.
We first set $n=k-1$ on both sides of \eqref{eqn:OmegaRec} to get:
\be \label{eqn:OmegaRecK}
\begin{split}
& \omega_0^{5/2} \pd_{w_0} \Omega_{g,k}(w_0,w_1, \dots, w_{k-1}) \\
 = & x^{5/2} \pd_x y^{5/2}\pd_y \Omega_{g-1,k+1}(x,y,w_{[k-1]})|_{x=y=w_0} \\
+ & \sum^s_{\substack{g_1+g_2=g \\ A_1 \coprod A_2 = [k-1]}}
w_0^{5/2}\pd_{w_0} \Omega_{g_1, |A_1|+1}(w_0,w_{A_1})
\cdot w_0^{5/2}\pd_{w_0} \Omega_{g_2, |A_2|+1}(w_0, w_{A_2}) \\
+ & \sum_{i=1}^{k-1} w_0\cD_{w_0,w_i} \pd_x \Omega_{g,k-1}(x, w_{[k-1]_i}).
\end{split}
\ee
We then set $w_0=\cdots = w_{k-1} = w$ on both sides of \eqref{eqn:OmegaRecK}.
The left-hand side becomes:
\ben
&& w^{5/2} \pd_w \Omega_{g,k}(w, \dots, w).
\een
The first line on the right-hand side becomes:
\ben
&& x^{5/2} \pd_xy^{5/2}\pd_y \Omega_{g-1,k+1}(x,y,w, \dots, w)|_{x=y=w}.
\een
The second line on the right-hand side becomes:
\ben
&&  w^{5/2}
\sum^s_{\substack{g_1+g_2=g \\ k_1 + k_2 = k+1}} \frac{(k-1)!}{k_1!k_2!}
\pd_{w} \Omega_{g_1, k_1}(w,\dots,w)
\cdot \pd_{w} \Omega_{g_2, k_2}(w_0, \dots, w).
\een
Now we deal with the third line on the right-hand side.
Recall
\be
\cD_{u,v} x^{a-1/2} = uv^{1/2}(u^{a+1} + u^av + \cdots + v^{a+1}).
\ee
So we have
\ben
w_0 \cD_{w_0,w_i} \pd_x x^{a+1/2}
& = & w_0 \cD_{w_0,w_i} (a+1/2) x^{a-1/2} \\
& = & (a+1/2) w_0 \cdot w_0w_i^{1/2}(w_0^{a+1}+w_0^aw_i + \cdots +w_i^{a+1}),
\een
it follows that
\ben
w_0 \cD_{w_0,w_i} \pd_x x^{a+1/2}|_{w_0=w_i=w}
& = & (a+2) (a+1/2) w \cdot ww^{1/2} w^{a+1} \\
& = & (a+2)(a+1/2) w^{a+7/2}.
\een
On the other hand
\ben
(w^{5/2}\pd_w)^2w^{a+1/2}
& = & w^{5/2}\pd_w(w^{5/2} \cdot (a+1/2) w^{a-1/2}) \\
& = & (a+1/2) w^{5/2} \pd_w w^{a+2} \\
& = & (a+1/2)(a+2) w^{5/2} w^{a+1} \\
& = & (a+1/2)(a+2) w^{a+7/2}.
\een
Hence we get:
\ben
&& \sum_{i=1}^{k-1} w_0\cD_{w_0,w_i} \pd_x \Omega_{g,k-1}(x, w_{[k-1]_i})|_{w_0=\cdots = w_{k-1} = w} \\
& = & (k-1) (x^{5/2}\pd_x)^2 \Omega_{g,k-1}(x,w,\dots, w)|_{x=w}.
\een
To summarize,
we have obtained the following identity:
\ben
&& w^{5/2} \pd_w \Omega_{g,k}(w, \dots, w) \\
&= & x^{5/2} \pd_xy^{5/2}\pd_y \Omega_{g-1,k+1}(x,y,w, \dots, w)|_{x=y=w} \\
& + &  w^{5/2}
\sum^s_{\substack{g_1+g_2=g \\ k_1 + k_2 = k+1}} \frac{(k-1)!}{k_1!k_2!}
\pd_{w} \Omega_{g_1, k_1}(w,\dots,w)
\cdot \pd_{w} \Omega_{g_2, k_2}(w_0, \dots, w) \\
& + & (k-1) (x^{5/2}\pd_x)^2 \Omega_{g,k-1}(x,w,\dots, w)|_{x=w}.
\een
Dividing both sides by $\frac{1}{(k-1)!}$ and take $\sum_{2g-1+k=n}$,
one gets \eqref{eqn:MainInW}.
This completes the proof.

\vspace{.1in}
{\em Acknowledgements}.
The author is partially supported by  NSFC grant 1171174.


\begin{thebibliography}{999}



\bibitem{Ben-Coc-Saf-Wos}
J. Bennett, D. Cochran, B. Safnuk, K. Woskoff,
{\em  Topological recursion for symplectic volumes of moduli spaces of curves},
arXiv:1010.1747.


\bibitem{DVV}
R. Dijkgraaf, H. Verlinde, and E. Verlinde,
{\em Topological strings in $d < 1$}, Nuclear Phys. B 352 (1991), 59-86.


\bibitem{Eyn-OraInv}
B. Eynard, N. Orantin,
{\em Invariants of algebraic curves and topological expansion},
Commun. Number Theory Phs. {\bf 1}(2007), no.2, 347-452,
arXiv:math-ph/0702045.

\bibitem{Eyn-OraRec}
B. Eynard, N. Orantin,
{\em Topological recursion in enumerative geometry and random matrices},
J. Phys. A: Math. Theor {\bf 42}(2009), 293001.

\bibitem{Guk-Sul}
S. Gukov, P. Su{\l}kowski,
{\em A-polynomial, B-model, and quantization},
JHEP1202(2012)070,
arXiv:1108.0002.


\bibitem{Kon}
 M. Kontsevich,
{\em Intersection theory on the moduli space of curves and the matrix Airy function},
Comm. Math. Phys. 147 (1992), no. 1, 1-23.

\bibitem{Wit}
E. Witten, {\em Two-dimensional gravity and intersection theory on moduli spaces},
Surveys in Differential Geometry,
vol.1 (1991), 243-269.


\bibitem{Zho}
J.~Zhou, {\em Topological recursions of Eynard-Orantin  type for intersection
numbers on moduli spaces of curves},
preprint, January, 2011.

\bibitem{Zho2}
J.~Zhou,
{\em Quantum mirror curves for $\bC^3$ and the resolved confiold},
in preparation.

\end{thebibliography}
\end{document}